\documentclass[12pt]{amsart}
\usepackage{amssymb}
\usepackage{amsmath}
\usepackage[english]{babel}

\newtheorem{thm}{Theorem}

\newtheorem{lemma}{Lemma}

\theoremstyle{definition}

\newcommand{\set}[1]{\left\{#1\right\}}
\newcommand{\norm}[1]{\left\Vert#1\right\Vert}
\newcommand{\cA}{\mathcal{A}}
\newcommand{\cB}{{\mathcal{B}}}
\newcommand{\cK}{{\mathcal{K}}}

\newcommand{\cG}{{\mathcal{G}}}

\newcommand{\cF}{{\mathcal{F}}}
\newcommand{\cH}{{\mathcal{H}}}
\newcommand{\cP}{{\mathcal{P}}}
\newcommand{\cX}{{\mathcal{X}}}

\newcommand{\cT}{{\mathcal{T}}}
\newcommand{\dC}{{\mathbb{C}}}
\newcommand{\cY}{{\mathcal{Y}}}
\newcommand{\dN}{{\mathbb{N}}}
\newcommand{\dT}{{\mathbb{T}}}

\newcommand{\dZ}{{\mathbb{Z}}}

\newcommand{\Ran}{\mathrm{Ran}}
\newcommand{\Ker}{\mathrm{Ker}}

\newcommand{\ov}{\overline}

\begin{document}
\title{A noncommutative version of the Fej\'er-Riesz theorem}

\author{Yuri\u\i{} Savchuk}
\address{Mathematisches Institut, Universit\"at Leipzig, Johannisgasse 26, 04103 Leipzig, Germany}
\email{savchuk@math.uni-leipzig.de}

\author{Konrad Schm\"udgen}
\address{Mathematisches Institut, Universit\"at Leipzig, Johannisgasse 26, 04103 Leipzig, Germany}
\email{schmuedgen@math.uni-leipzig.de}




\begin{abstract}
Let $\cX$ be the unital $*$-algebra generated by the unilateral shift operator. It is shown that for any nonnegative operator $X\in \cX$ there is an element $Y\in \cX$ such that $X=Y^*Y$.
\end{abstract}

\maketitle

\section{Introduction and Main Result}
Let $\cP$ denote the $*$-algebra of complex Laurent polynomials $p(z,z^{-1})\\=\sum_{k=-n}^{n} a_k z^k$ with involution $p \to \overline{p}(z):= \sum_{k=-n}^{n} \ov{a}_k z^{-k},$ where $n\in\dN_0.$ Setting $z=e^{it}$ with $t \in [0,2\pi]$ we see that $\cP$ is isomorphic to the $*$-al\-geb\-ra of all trigonometric polynomials. There is a faithful $*$-re\-pre\-sen\-ta\-ti\-on $\pi$ of the $*$-algebra $\cP$ on the Hilbert space $l^2(\dZ)$ such that $\pi(z)=U$, where $U$ is the bilateral shift operator.

The classical Fej\'er-Riesz theorem states that if a polynomial $p \in \cP$ takes only nonnegative values on the unit circle $\dT=\{z\in \dC:|z|{=}1\}$ or equivalently if the operator $\pi(p)$ on $l^2(\dZ)$ is nonnegative, then $p$ is of the form $p=\ov{q}\cdot{q}$ for some $q \in \cP$. Further, if $p$ has degree $d\in \dN_0$, then $q$ can be chosen to be an analytic polynomial $q(z)=\sum_{k=0}^d b_{k}z^{k}$ of degree $d$ such that $q(z)\not= 0$ for $|z|<1$ and $q(0)>0.$ The latter conditions determine the polynomial $q$ uniquely.
A simple proof of this theorem 
can be found in \cite{sz}, see Theorems 1.2.1 and 1.2.2 therein.

The aim of this paper is to prove an analog of the Fej\'er-Riesz theorem if $\cP$ is replaced by the unital $*$-algebra $$\cA=\dC\langle s,s^*\ |\ s^*s=1\rangle$$ and the bilateral shift $U$ is replaced by the unilateral shift
\begin{align}\label{unlat}
S(\varphi_{0},\varphi_1,\varphi_{2},\dots)=(0,\varphi_{0},\varphi_{1},\varphi_2,\dots)
\end{align}
on the Hilbert space $l^2(\dN_0)$. Let $\pi_0$ denote the $*$-representation of $\cA$ on $l^2(\dN_0)$ determined by $\pi_0(s)=S$.
Our main result is the following 
\begin{thm}
For any element $x=x^*\in\cA$ the following statements are equivalent:\\
(i) $x=y^*y$ for some $y\in\cA.$\\
(ii) $\pi(x)\geq 0$ for any $*$-representation of the $*$-algebra $\cA$ on a Hilbert space.\\
(iii) $\pi_0(x)\geq 0$ on the Hilbert space $l^2(\dN_0)$.\\
If this holds, then $y$ can be chosen such that the matrix of $\pi_0(y)$ with respect to the standard base of  $l^2(\dN_0)$ is lower-triangular.
\end{thm}
The implications (i)$\to$ (ii)$\to$ (iii) are trivial, so it remains to prove that (iii) implies (i). This will be done in the next section. 

If $\cH$ and $\cK$ are Hilbert spaces, $\cB(\cH,\cK)$ are the bounded operators from $\cH$ into $\cK$ and $\cB(\cH){:=}\cB(\cH,\cH).$ For $x,y \in \cH,\ x\otimes y$ denotes the rank one operator $\langle\cdot, x\rangle y$. By a $*$-representation of a unital $*$-algebra on $\cH$ we mean a unit preserving $*$-homomorphism into the $*$-algebra $\cB(\cH)$.

\section{Proof of the Main Implication}
For the main proof we need three simple lemmas. The second lemma is a well-known fact on outer analytic polynomials, while the third is the crucial factorization lemma. To make the exposition in this section as elementary as possible we include complete proofs.

We  identify the Hilbert spaces $l^2(\dN_0)$ and $H^2(\dT)$ in the obvious way by identifying their standard orthonormal bases $\set{e_k;k \in \dN_0}$ and $\set{z^k;\ k\in\dN_0}.$ Let $\cT_{p}$ denote the set of all Toeplitz operators $T_p$ on $H^2(\dT)$ with symbol $p \in \cP$ and $\cF$ the set of all bounded operators on $l^2(\dN_0)$ which have finite matrices with respect to the base $\set{e_k}.$ That is, $F\in\cB(l^2(\dN_0))$ is in $\cF$ if and only if there exists a natural number $k$ such that $\langle Fe_i,e_j\rangle=0$ if $i>k$ or $j>k.$

Set $\cX:=\pi_0(\cA).$
\begin{lemma}\label{lemma_atf}
$\cX=\cT_p+\cF.$ 
\end{lemma}
\begin{proof}
Put $\cY:= \cT_p +\cF$. 
Since $\pi_0(z^n){=}\pi_0(z)^n {=}S^n=T_{z^n}$ and\\ $\pi_0(z^{-n}){=}\pi_0((z^n)^*)){=}(T_{z^n})^*{=}T_{z^{-n}}$ for $n \in \dN_0$, we have $\cT_{p} \subseteq \cX$. From the relations
$\pi_0(s^n(1-ss^*)s^{*k})=S^n(I-SS^*)S^{*k}= e_k \otimes e_n$, $k, n \in \dN_0$, we conclude that $\cF \subseteq \cX$. Thus, $\cY=\cT_{p} +\cF \subseteq \cX$.

Now we prove the converse inclusion $\cX \subseteq \cY$. For $n \in \dN$ we have $ST_{z^n}=T_{z^{n+1}}$, $S^*T_{z^n}=T_{z^{n-1}}$, $ST_{z^{-n}}=T_{z^{1-n}} - e_{n-1}\otimes e_0$, and $S^*T_{z^{-n}}=T_{z^{-n-1}}$. These relations imply that $S\cdot \cT_p \subseteq \cY$ and $S^* \cdot \cT_p \subseteq \cY$. Since obviously $S\cdot \cF \subseteq \cF$ and $S^* \cdot \cF \subseteq \cF$, we obtain $S\cdot \cY \subseteq \cY$ and $S^*\cdot \cY \subseteq \cY$. Because $\cX$ is generated as an algebra by $S=\pi_0(z)$ and $S^*=\pi_0(z^{-1})$, the latter yields $\cX\cdot \cY\subseteq \cY$, so  $\cX \subseteq \cY$. Consequently, $\cX=\cY.$
\end{proof} 
\begin{lemma}\label{randense}
Let $q(z)=\sum_{k=0}^d b_kz^k$ be an analytic polynomial such that $q(z)\not= 0$ for $|z|<1$. Then $\Ran T_q =q(z)H^2(\dT)$ is dense in $H^2(\dT)$.
\end{lemma}
\begin{proof}
We proceed by induction on the degree of $q$. Suppose that the assertion holds when $\deg q\leq d$. Let $q_0(z)=\sum_{k=0}^{d{+}1} b_kz^k$ be as above and $\deg q_0={d{+}1}.$ We write $q_0(z)=(z{-}\lambda)q(z).$ Then $q(z)$ satisfies also the assumptions, so $qH^2(\dT)$ is dense by the induction hypothesis. Therefore it suffices to show that $\Ran(S{-}\lambda I)=(z{-}\lambda)H^2(\dT)$ is dense or equivalently that $\Ker(S^*{-}\ov{\lambda} I)=\{0\}$. For let $\varphi{=}(\varphi_n)\in\Ker(S^*{-}\ov{\lambda} I)$. Then we have $\varphi_n-\ov{\lambda}\varphi_{n{-}1}=0$ and hence $\varphi_n=\ov{\lambda}^n\varphi_0$ for $n\in \dN$. Since $q_0(z)\not= 0$ for $|z|<1$, $|\lambda|\geq 1$ and hence $\varphi=0$, because $\varphi \in l^2(\dN_0)$.
\end{proof}
\begin{lemma}\label{lemma_ABC}
Let $\cH$ and $\cK$ be Hilbert spaces, $A \in \cB(\cH)$, $W\in\cB(\cK)$, and $V\in\cB(\cH,\cK)$. Let $P_W$ denote the orthogonal projection of $\cK$ onto the closure of $\Ran W$. Suppose that the block matrix
\begin{gather*}
\left(\begin{array}{ll}
\ \ A & V^*W \\ 
W^*V & W^*W
\end{array}\right)
\end{gather*}
defines a nonnegative operator on $\cH\oplus \cK$. Then we have $A\geq V^*P_W V.$\\ For any $U\in\cB(\cH)$ such that $A-V^*P_W V=U^*U$, we have
\begin{gather}\label{factorau}
\left(\begin{array}{ll}
\ \ A & V^*W \\ 
W^*V & W^*W
\end{array}\right)=
\left(\begin{array}{ll}
\ \ U & 0 \\ 
P_W V & W
\end{array}\right)^*
\left(\begin{array}{ll}
\ \ U & 0 \\ 
P_W V & W
\end{array}\right).
\end{gather}

\end{lemma}
\begin{proof}
Fix $\varphi\in\cH$ and let $\psi\in\cK$. Since the block matrix is nonnegative, for all $\lambda \in \dC$ we have the inequality 
\begin{gather*}
\left\langle\left(\begin{array}{ll}
\ \ A & W^*V \\ 
V^*W & W^*W
\end{array}\right)\left(\begin{array}{l}\ \varphi \\ \lambda\psi \end{array}\right),\left(\begin{array}{l}\ \varphi \\ \lambda\psi \end{array}\right)\right\rangle\geq 0
\end{gather*}
which can be written as 
\begin{gather}\label{alambda}
\langle A\varphi,\varphi\rangle+\lambda\langle V^*W\psi,\varphi\rangle+\overline{\lambda}\langle W^*V\varphi,\psi\rangle+\lambda\overline{\lambda}\langle W^*W\psi,\psi\rangle\geq 0. 
\end{gather}
Since the latter inequality holds for {\it arbitrary} $\lambda\in\dC$, we conclude that
\begin{gather*}
\langle A\varphi,\varphi\rangle\langle W^*W\psi,\psi\rangle=\langle A\varphi,\varphi\rangle \norm{ W\psi}^2\geq|\langle V^*W\psi,\varphi\rangle|^2=|\langle W\psi,V\varphi\rangle|^2, 
\end{gather*}
By the definition of $P_W$, there is a sequence $\psi_n\in\cK$, $n \in\dN$, such that $W\psi_n \to P_W V\varphi$. Setting $\psi=\psi_n$ in the preceding inequality and passing to the limit we obtain 
\begin{gather*}
\langle A\varphi,\varphi\rangle\norm{P_WV \varphi}^2\geq | \langle P_W V\varphi,V\varphi\rangle|^2=\norm{ P_W V\varphi}^4.
\end{gather*}
Hence $\langle A \varphi,\varphi \rangle \geq \langle V^*P_W V \varphi ,\varphi \rangle$ when $P_WV\varphi \not= 0$.
Since $\langle A \varphi,\varphi \rangle \geq 0$ by setting $\lambda =0$ in (\ref{alambda}), we have $\langle A\varphi, \varphi\rangle \geq 0= \langle V^*P_W V \varphi ,\varphi \rangle$ when $P_W P\varphi =0$. Therefore, $A\geq V^*P_WV.$ Equation (\ref{factorau}) is obvious.
\end{proof}

Now we are ready to prove the implication (iii)$\to$(i). 

First we note that by Lemma \ref{lemma_atf} a bounded operator $X$ on $l^2(\dN_0)$ belongs to $\pi_0(\cA)$ if and only if it has a matrix representation
\begin{gather}\label{eq_repX}
X=\left(
\begin{array}{lllllll}
x_{00} & x_{01} & \dots & x_{0,n} & 0 & \dots & \dots \\ 
x_{10} & x_{11} & \dots & x_{1,n} & x_{-n} & 0 & \dots \\ 
× & × & \ddots & × & \vdots & x_{-n} & \ddots \\ 
x_{n,0} & x_{n,1} & \dots & x_{n,n} & x_{-1} & \vdots & \ddots \\ 
0 & x_n & \dots & x_1 & x_0 & x_{-1} & × \\ 
\vdots & 0 & x_n & \dots & x_1 & x_0 & \ddots \\ 
\vdots & \vdots & \ddots & \ddots & × & \ddots & \ddots
\end{array}\right)
\end{gather}
with respect to the standard basis $\{e_k\}$ of $l^2(\dN_0).$ (If $X=F +T_p$ with $F\in \cF$ and $p\in \cP$, by adding zeros we can find a common $n$ such that $p=\sum_{k{=}-n}^n a_kz^k$ and $F$ has the size  $(n{+}1) \times  (n{+}1)$.) For simplicity we use the same notation for operators and the corresponding matrices.

Suppose that $x{=}x^* \in \cA$ and $X:=\pi_0(x)\geq 0$. Let (\ref{eq_repX}) be the matrix  of $X$. Since $X$ is symmetric, we have $x_{ij}=\overline{x_{ji}}$ and $x_k=\overline{x_{-k}}$ for all $i,j,k.$ We will prove that there is a lower-triangular matrix
\begin{gather}\label{eq_repY}
Y=\left(
\begin{array}{lllllll}
y_{00} 	&	 0 &		&  	  	&  	&  	& × \\ 
y_{10} 	& y_{11}  &  0		&  	  	&  	&  	& × \\ 
\vdots 	& \vdots  & \ddots 	&	\ddots	&  	&  	&   \\ 
y_{n,0} & y_{n,1} & \dots	& y_{n,n} 	& 0 	&  	&   \\ 
0 	& y_n 	  & \dots	& y_1 		& y_0 	& 0 	& × \\ 
\vdots 	& 0 	  & y_n		& \dots		& y_1 	& y_0 	& 0  \\ 
× 	& \vdots  & \ddots	& \ddots	& × 	& \ddots& \ddots 
\end{array}\right)
\end{gather}

\noindent such that $X=Y^*Y$. Since the matrix (\ref{eq_repY}) of $Y$ is also of the form (\ref{eq_repX}), we conlude that $Y\in\cX$.

Let $P_n$ be the projection of $l^2(\dN_0)$ onto the linear span of $e_0,\dots,e_n.$ Then we write $X$ and $Y$ as block matrices
\begin{gather}\label{eq_Y=UVW}
X=\left(\begin{array}{ll}
A & B \\ 
B^* & C
\end{array}\right)\ \mbox{and}\ 
Y=\left(\begin{array}{ll}
U & 0 \\ 
V & W 
\end{array}\right)
\end{gather} 
where the blocks $A,B,C$ and $U,V,W$ correspond to the matrices of $P_nXP_n$, $ (I-P_n)XP_n$, $(I-P_n)X(I-P_n)$ and $P_nYP_n$, $(I-P_n)YP_n$, $(I-P_n)Y(I-P_n)$, respectively.

Define $p(z){:=}\sum_{k=-n}^n x_kz^k\in \cP$. If $p\equiv 0$, then $X$ has the positive semi-definite matrix $A$ in the left-upper corner and zeros elsewhere. By the Cholesky UL-decomposition (see e.g. \cite{hous}, p.13) we have $A=U^*U$ for some lower-triangular matrix $U.$ Putting $V=0,W=0$ in (\ref{eq_Y=UVW}) the assertion is proven in this case. 

Assume now that $p\not\equiv 0$. The assumption $X\geq 0$ implies that $C=(I-P_n)X(I-P_n)\geq 0$. By (\ref{eq_repX}), $C$ 
has the same matrix as the Toeplitz operator $T_p$, so that $T_p \geq 0$. Since nonnegative Toeplitz operators have nonnegative symbols (see e.g. \cite{do}, 7.19), it follows that $p(z)\geq 0$ for all $z \in \dT.$ Therefore, by the Fej\'er-Riesz theorem there is a polynomial $q(z)=\sum_{k=0}^n y_kz^k \in \cP$ such that $p=\ov{q}q$, $q(z)\not= 0$ for $|z|<1$ and $q(0)>0.$ (Note the the degrees of $p$ and $q$ may be smaller than $n$.)

Having the polynomial $q(z)=\sum_{k=0}^n y_kz^k,$ we define $V$ and $W$ as above (see (\ref{eq_repY})). Then $W^*W=T_{\ov{q}}T_q=T_p$. A direct computation shows that $V^*W=B.$ Since $q(z)\not= 0$ for $|z|<1$, it follows from Lemma \ref{randense} that $\Ran T_q \equiv\Ran W$ is dense in the corresponding Hilbert space $(I-P_n)H^2(\dT)$, so that $P_W=I$. Therefore, $A\geq V^*V=V^*P_WV$ by Lemma \ref{lemma_ABC}. Applying the Cholesky UL-decomposition to the finite positive semi-definite matrix $A-V^*V$, it follows that there is a lower-triangular matrix $U$ such that $A-V^*V=U^*U$. Then we have $X=Y^*Y.$

Since $Y\in\cX,$ there is a $y\in\cA$ such that $Y=\pi_0(y).$ Then we have $\pi_0(x)=X=Y^*Y=\pi_0(y)^*\pi_0(y)=\pi_0(y^*y).$ Since the representation $\pi_0$ of $\cA$ is faithful, it follows that $x=y^*y$ and statement (i) is proven.

\section{Concluding Remarks}

1. For any finite positive semi-definite complex matrix there are an LU-decomposition and an UL-decomposition, see \cite{golub},\cite{hous}. The UL-decomposition was used in the above proof. It might be worth to emphasize that a direct generalization of the LU-decomposition algorithm would not work to prove our result. For let 
\begin{gather*}
Y_1{=}\left(
\begin{array}{cccccc}
  3 	&\ 2 	&   	&   	&   	&   \\
  1 	&\ 1 	&\ 0 	&   	&   	&   \\
  0  	&\ 1 	&\ 1 	& 0 	&   	&   \\
    	&\ 0  	&\ 1 	& 1 	& 0 	&   \\
    	&   	&\ 0  	& 1 	& 1 	& \ddots \\
    	&   	&   	&\ddots &\ddots & \ddots
\end{array}\right),\ 
X{=}\left(
\begin{array}{cccccc}
  10 	&\ 7 	&\ 0  	&   	&   	&   	 \\
  7 	&\ 6 	&\ 1 	& 0  	&   	&   	 \\
  0  	&\ 1 	&\ 2 	& 1 	& 0  	&   	 \\
    	&\ 0  	&\ 1 	& 2 	& 1 	& \ddots \\
    	&   	&\ 0  	& 1 	& 2 	& \ddots \\
    	&   	&   	&\ddots &\ddots & \ddots
\end{array}
\right).
\end{gather*}
Then we have $X=Y_1^*Y_1$ and hence $X\geq 0$. By Lemma \ref{lemma_atf}, $Y_1$ and $X$ are in the $*$-algebra $\cX.$ Applying the row-by-row algorithm for the LU-decomposition (see e.g. \cite{golub}) to the infinite matrix $X$ we obtain 
$$
Y_2=\left(
\begin{array}{cccccc}
  \sqrt{10} & \frac{7}{\sqrt{10}} & 0  &   &   &   \\
  0 & \sqrt{\frac{11}{10}} & \sqrt{\frac{10}{11}} & 0  &   &   \\
    & 0 & \sqrt{\frac{12}{11}} & \sqrt{\frac{11}{12}} & 0  &   \\
    &   & 0 & \sqrt{\frac{13}{12}} & \sqrt{\frac{12}{13}} & \ddots  \\
    &   &   & 0 & \sqrt{\frac{14}{13}} & \ddots \\
    &   &   &   & \ddots & \ddots
\end{array}
\right)
$$
and $X=Y_2^*Y_2^{}.$ But Lemma \ref{lemma_atf} implies that $Y_2$ is not in $\cX$! It was crucial for our result  to get a factorization {\it inside} the $*$-algebra $\cX$. However the UL-construction  yields the decomposition $X=Y^*Y,$ where 
$$
Y=\left(\begin{array}{cccccc}
  \frac{1}{\sqrt{5}} & 0 &   &   &   &   \\
  \frac{7}{\sqrt{5}} & \sqrt{5} & 0 &   &   &   \\
   0 & 1 & 1 &\ 0 &   &   \\
     & 0 & 1 &\ 1 & 0 &  \\
     &   & 0 &\ 1 & 1 & \ddots \\
     &   &   & \ddots  & \ddots & \ddots
\end{array}
\right)\in \cX.
$$

2. The main result of \cite{hmp} implies that each nonnegative $X\in \cX$ is a {\it finite sum} of hermitian squares $Y^*Y$ in the $*$-algebra $\cX$. Our theorem says that such an $X$ is a {\it single} hermitian square. Results of this kind can be interpreted in the context of noncommutative real algebraic geometry \cite{s}.

3. There is the following operator version of our main result:\\
{\it Let $\cH\not= \{0\}$ be a Hilbert space and let $\cX$ denote the set of all bounded operators on the Hilbert space $\cK=\oplus_{k=0}^\infty \cH_k$, where $\cH_k{=}\cH$, which are given by operator block matrices of the form (\ref{eq_repX}) with entries $x_{ij}$ and $x_k$ from $\cB(\cH)$.  
Then, for any element $X=X^*\in \cX$ such that $X\geq 0$ on $\cK$ there is an element $Y\in \cX$ such that $X=Y^*Y.$ Moreover, $Y \in \cX$ can be chosen as a lower-triangular block matrix. }

A proof of this result can be given along the lines of the proof in Section 2 with the following modifications. 
All bars of numbers are replaced by the adjoints of operators. The Hilbert space $\cK$ is identified with the Hardy space $H^2_\cH(\dT)$ of $\cH$-valued functions. Then the operator Laurent polynomial $p(z)=\sum_{k=-n}^n x_k z^k$
is nonnegative on $\cH$ for all $z \in \dT$. Therefore, by Rosenblum's operator Fej\'er-Riesz theorem (\cite{ro}, see e.g. \cite{dr} for a nice approach) there exists an operator-valued outer function $q(z)=\sum_{k=0}^n y_k z^k$, where $y_0,\dots,y_n \in \cB(\cH)$, such that $p(z)=q(z)^*q(z)$ for all $ z\in \dT$. That $q$ is outer means that there is a closed subspace $\cG$ of $\cH$ such that $H^2_\cG(\dT)$ is the closure of the range of the multiplication operator by $q$ on $H^2_\cH(\dT)$. If $P_\cG$ denotes the projection of $\cH$ onto $\cG$, then the projection $P_W$ from Lemma \ref{lemma_ABC} acts as $P_W(\varphi_n)=(P_\cG \varphi_n)$. Therefore, the operator matrix $Y$, defined by (\ref{eq_Y=UVW}) with $V$ replaced by $P_W V$, is in $\cX$. By Lemma \ref{lemma_ABC} the finite block matrix $A-V^*P_WV$ on the n-fold sum $\cH\oplus \cdots \oplus\cH$ is nonnegative. Hence there exists a lower triangular block matrix $U$ such that $A-V^*P_W V=U^*U$ (see \cite{ff}, Remark 7.5). Then we have $X=Y^*Y$ by (\ref{factorau}) which completes the proof.

\bibliographystyle{amsalpha}

\begin{thebibliography}{10}



\bibitem[Do]{do} R.Douglas, \textit{Banach algebra techniques in operator theory}, Academic Press, New Yorck, 1972.

\bibitem[D]{dr} M.Dritschel, \textit{On factorization of trigonometric polynomials}, Integral Equ. Operator Theory {\bf{49}} (2004), 11-42.

\bibitem[FF]{ff} C.Foias, Frazho, A., \textit{The commutant lifting approach to interpolation problems}, Birkh\"auser Verlag, Basel, 1990.

\bibitem[G]{golub} G.H. Golub, Van Loan, C.F., \textit{Matrix computations}, Johns Hopkins University Press, Baltimore, MD, 1996.

\bibitem[H]{hous} A.S. Householder, \textit{The theory of matrices in numerical analysis}, Blaisdell Publishing Co. Ginn and Co., New York-Toronto-London, 1964.

\bibitem[HMP]{hmp} J. Helton, McCullough, S., Putinar, M., \textit{A noncommutative Positivstellensatz on isometries}, J. Reine Angew. Math. {\bf{568}} (2004), 71--80. 

\bibitem[R]{ro} M. Rosenblum, \textit{Vectorial Toeplitz operators and the Fejer-Riesz theorem}, J. Math. Anal. Appl. {\bf{23}}(1968), 139--147.

\bibitem[S]{s} K. Schm\"udgen, \textit{Noncommutative real algebraic geometry -- some basic concepts and first ideas}, In: Emerging applications of algebraic geometry. M. Putinar and S. Sullivant (eds.), Springer-Verlag, 2009.

\bibitem[Sz]{sz} G. Szeg{\H{o}}, \textit{Orthogonal polynomials}, Colloquium Publications, Vol. XXIII, Amer. Math. Soc., Providence, R.I., 1975.


\end{thebibliography}

\end{document}